\documentclass{amsart}
\usepackage[mathscr]{euscript}

\input xy
\xyoption{all}

\newcommand{\F}{\mathcal F}
\newcommand{\K}{\mathcal K}
\newcommand{\C}{\mathscr C}
\newcommand{\IR}{\mathbb R}
\newcommand{\IN}{\mathbb N}

\newcommand{\w}{\omega}
\newcommand{\osc}{\mathrm{osc}}
\newcommand{\dec}{\mathrm{dec}}

\newcommand{\cov}{\mathrm{cov}}
\newcommand{\M}{\mathcal M}
\newcommand{\wid}{\mathrm{wd}}
\newcommand{\U}{\mathcal U}
\newcommand{\e}{\varepsilon}

\newtheorem{theorem}{Theorem}[section]
\newtheorem{proposition}[theorem]{Proposition}
\newtheorem{corollary}[theorem]{Corollary}
\newtheorem{problem}[theorem]{Problem}
\newtheorem{example}[theorem]{Example}

\title[On $\infty$-convex sets in function spaces]{On $\infty$-convex sets in spaces of scatteredly continuous functions}
\author{Taras Banakh, Bogdan Bokalo, and Nadiya Kolos}%author's name
\address{T.Banakh: Ivan Franko National University of Lviv (Ukraine) and Jan Kochanowski University in Kielce (Poland)}
\email{t.o.banakh@gmail.com}
\address{B.Bokalo, N.Kolos: Ivan Franko National University of Lviv}%author's university
\email{bogdanbokalo@mail.ru, nadiya\_kolos@ukr.net} %author's e-mail address
\thanks{The first author has been partially financed by NCN means granted by decision DEC-2011/01/B/ST1/01439}
\dedicatory{Dedicated to Mitrofan Choban and Stoyan Nedev on the occasion of their 70th birthdays}

\begin{document}

\begin{abstract} Given a topological space $X$, we study the structure of $\infty$-convex subsets in the space $SC_p(X)$ of scatteredly continuous functions on $X$. Our main result says that for a topological space $X$ with countable strong fan tightness, each potentially bounded $\infty$-convex subset $\F\subset SC_p(X)$ is weakly discontinuous in the sense that each non-empty subset $A\subset X$ contains an open dense subset $U\subset A$ such that each function $f|U$, $f\in\F$, is continuous.
This implies that $\F$ has network weight $nw(\F)\le nw(X)$.
\end{abstract}
\maketitle

\section{Introduction}

In this paper we continue the study of the linear-topological
structure of the function spaces $SC_p(X)$, started in
\cite{BK1,BK2,BBK}. Here $SC_p(X)$ stands for the space of all
scatteredly continuous real-valued functions on a topological space
$X$, endowed with the topology of pointwise convergence. A function
$f:X\to Y$ between two topological spaces is called {\em scatteredly
continuous} if for each non-empty subset $A\subset X$ the
restriction $f|A$ has a point of continuity\footnote{The term ``{\em scatteredly continuous function}''  was suggested by Mitrofan Choban after a talk of the second author at the conference dedicated to 20th anniversary of the Chair of Algebra and Topology of Ivan Franko National University of Lviv that was held on September 28, 2001 in Lviv (Ukraine).}. By \cite{AB} or
\cite[4.4]{BB}, each scatteredly continuous function $f:X\to Y$ with
values in a regular topological space $Y$ is {\em weakly
discontinuous} in the sense that each non-empty subset $A\subset X$
contains an open dense subset $U\subset A$ such that the restriction
$f|U$ is continuous. This implies that $SC_p(X)$ is a linear space.

The space $SC_p(X)$ contains the linear subspace $C_p(X)$ of all continuous real-valued functions on $X$. By its topological properties the function spaces $SC_p(X)$ and $C_p(X)$ differ substantially. For example, for an uncountable metrizable separable space $X$, the function space $SC_p(X)$ contains non-metrizable compacta while $C_p(X)$ does not.
In spite of this difference, in \cite{BBK} it was observed that compact {\em convex} subsets of the space $SC_p(X)$ share many properties with compact convex subsets of the space $C_p(X)$. In particular, for a space $X$ with countable network weight, all compact convex subsets of the function space $SC_p(X)$ are metrizable. This follows from Corollary~1 of \cite{BBK} saying that for a countably tight topological space $X$, each $\sigma$-convex subset $\F\subset SC_p(X)$ has network weight $nw(\F)\le nw(X)$.

A subset $C$ of a linear topological space $L$ is called {\em $\sigma$-convex} if for any sequence $(x_n)_{n\in\w}$ of points of $K$ and any sequence $(t_n)_{n\in\w}$ of non-negative real numbers with $\sum_{n=0}^\infty t_n=1$ the series $\sum_{t=0}^\infty t_nx_n$ converges to a point of the set $K$.
It is easy to see that each compact convex subset $K$ of a locally convex linear topological space $L$ is $\sigma$-convex, and each $\sigma$-convex subset of $L$ is convex and bounded in $L$.
We recall that a subset $B$ of a linear topological space $L$ is {\em bounded} in $L$ if for each neighborhood $U\subset L$ of zero there is a number $n\in\IN$ such that $B\subset nU$.

A more general notion comparing to the $\sigma$-convexity is that of
$\infty$-convexity. Following \cite{BLM}, we define a subset $K$ of
a linear topological space $L$ to be {\em $\infty$-convex} if for
each bounded sequence $(x_n)_{n\in\w}$ in $K$ and each sequence
$(t_n)_{n\in\w}$ of non-negative real numbers with
$\sum_{n=0}^\infty t_n=1$ the series $\sum_{t=0}^\infty t_nx_n$
converges to a point of the set $K$. In Proposition~\ref{p5.1} we
shall prove that a subset $K$ of a  linear topological space $L$ is
$\sigma$-convex if and only if it is bounded and $\infty$-convex.

In this paper we shall study the structure of $\infty$-convex subsets in the spaces $SC_p(X)$ of scatteredly continuous functions on a topological space $X$. A principal question, which will be asked about $\infty$-convex subsets $\F\subset SC_p(X)$ is their weak discontinuity. A function family $\F\subset SC_p(X)$ is called {\em weakly discontinuous} if each subset $A\subset X$ contains a dense open subset $U\subset A$ such that for each function $f\in \F$ the restriction $f|U$ is continuous. It is easy to see that a function family $\F\subset SC_p(X)$ is weakly discontinuous if and only if its diagonal product $\Delta\F:X\to\IR^\F$, $\Delta\F:x\mapsto (f(x))_{f\in\F}$, is weakly discontinuous.

In \cite{BBK} we proved that for each countably tight space $X$,
each $\sigma$-convex subset $\F$ of the space $SC_p^*(X)$ of bounded
scatteredly continuous functions on $X$ is weakly discontinuous and
has network weight $nw(\F)\le nw(X)$. Our first theorem is the
generalization of this fact to $\infty$-convex subsets of the
function space $SC_p^*(X)$.

\begin{theorem}\label{t1.1} If a topological space $X$ has countable tightness, then each $\infty$-convex subset $\F\subset SC_p^*(X)$ is weakly discontinuous and has network weight $nw(\F)\le nw(X)$.
\end{theorem}

Let us recall that a topological space $X$ has {\em countable tightness} if for each subset $A\subset X$ and a point $a\in\bar A$ in its closure there is an at most countable subset $B\subset A$ with $a\in\bar B$.

Observe that Theorem~\ref{t1.1} treats $\infty$-convex subsets of the function space $SC_p^*(X)$. For topological spaces $X$ with countable strong fan tightness, this theorem remains true for potentially bounded $\infty$-convex subsets of the function spaces $SC_p(X)$.

Following \cite{Sakai} we shall say that a topological space $X$ has {\em countable strong fan tightness} if for each sequence $(A_n)_{n\in\w}$ of subsets of $X$ and a point $a\in\bigcap_{n\in\w}\bar A_n$ there is a sequence of points $a_n\in A_n$, $n\in\w$, such that $a$ lies in the closure of the set $\{a_n\}_{n\in\w}$. The class of spaces of countable strong fan tightness and some related classes will be discussed in Section~\ref{s:fan}.

A subset $B\subset L$ of a linear topological space $L$ is called {\em potentially bounded} if for each sequence $(x_n)_{n\in\w}$ of points of $B$ there is a sequence $(t_n)_{n\in\w}$ of positive real numbers such that the set $\{t_nx_n\}_{n\in\w}$ is bounded in $L$. For example, the set $SC_p^*(X)$ is potentially bounded in the function space $SC_p(X)$.

\begin{theorem} If a topological space $X$ has countable strong fan tightness, then each potentially bounded $\infty$-convex subset $\F\subset SC_p(X)$ is weakly discontinuous and has network weight $nw(\F)\le nw(X)$.
\end{theorem}

This theorem follows from a more general Theorem~\ref{t9.1} and Corollary~\ref{c4.3}.
The potential boundedness is a necessary condition in this theorem because for a $\sigma$-compact perfectly normal space $X$, each weakly discontinuous family $\F\subset SC_p(X)$ is potentially bounded, see Corollary~\ref{c7.2}.

\section{Cardinal Invariants of topological spaces} In this section we recall the definition of some cardinal invariants that will appear in the next sections. More information on cardinal invariants of topological spaces can be found in \cite{En} and \cite{Juh}.

Let us recall that for a topological space $X$ its
\begin{itemize}
\item {\em weight} $w(X)$ is the smallest cardinality of a base of the topology of $X$;
\item {\em network weight} $w(X)$ is the smallest cardinality of a network of the topology of $X$;
\item {\em character} $\chi(X)$ is the smallest infinite cardinal $\kappa$ such that each point $x\in X$ has a neighborhood base of cardinality $\le\kappa$;
\item {\em tightness} $t(X)$ is the smallest infinite cardinal $\kappa$ such that for each subset $A\subset X$ and a point $a\in \bar A$ in its closure there is a subset $B\subset A$ of cardinality $|B|\le\kappa$ such that $a\in \bar B$;
\item {\em cellularity} $c(X)$ is the smallest infinite cardinal $\kappa$ such that no disjoint family $\U$ of open subsets of $X$ has cardinality $|\U|>\kappa$;
\item {\em spread} $s(X)=\sup\{c(Z):Z\subset X\}=\sup\{|D|:D$ is a discrete subspace of $X\}$;
\item {\em Lindel\"of number} $l(X)$ is the smallest infinite cardinal $\kappa$ such that each open cover of $X$ has a subcover of cardinality $\le\kappa$;
\item {\em hereditary Lindel\"of number} $hl(X)=\sup\{l(Z):Z\subset X\}$;
\item {\em density $d(X)$} if the smallest cardinality of a dense subset of $X$;
\item {\em hereditary density} $hd(X)=\sup\{d(Z):Z\subset X\}$;
\item {\em compact density} $kd(X)$ is the smallest cardinality $\C$ of a family of compact subsets of $X$ with dense union $\bigcup\C$;
\item {\em closed hereditary compact density} $h_{cl}kd(X)=\sup\{kd(F):F$ is a closed subspace of $X\}$.
\end{itemize}

It is easy to see that each topological space $X$ has closed hereditary density $h_{cl}kd(X)\le hd(X)$ and compact spaces $X$ have $kd(X)=h_{cl}kd(X)=1$.

\section{Decomposition numbers of function families}

For topological spaces $X,Y$ and a function family $\F\subset Y^X$
its ({\em closed}) {\em decomposition number} $\dec(\F)$ (resp.
$\dec_{cl}(\F)$) is defined as the smallest cardinality $|\C|$ of a
cover of $X$ by (closed or finite) subsets such that for each set
$C\in\C$ and a function $f\in\F$ the restriction $f|C$ is
continuous. It is clear that $\dec(\F)\le\dec_{cl}(\F)\le|X|$.

\begin{proposition}\label{p3.1} For any topological space $X$, each function family $\F\subset \IR^X$ has network weight $nw(\F)\le dec(\F)\cdot nw(X)$.
\end{proposition}

\begin{proof} Let $\C$ be a cover $X$ of cardinality $|\C|=\dec(\F)$ such that for each set $C\in\C$ and a function $f\in\F$ the restriction $f|C$ is continuous. Let $Z=\oplus\C$ be the topological sum of the spaces $C\in\C$ and $\pi:Z\to X$ be the natural projection, which induces a linear topological embedding $\pi^*:\IR^X\to\IR^Z$, $\pi^*:f\mapsto f\circ \pi$. The choice of the cover $\C$ guarantees that $\pi^*(\F)\subset C_p(Z)$. By Theorem~I.1.3 of \cite{Arh}, $nw(C_p(Z))\le nw(Z)$. Consequently,
$$nw(\F)=nw(\pi^*(\F))\le nw(C_p(Z))\le nw(Z)=\sum_{C\in\C}nw(C)\le |\C|\cdot nw(X)=\dec(\F)\cdot nw(X).$$
\end{proof}

For compact function families $\F\subset \IR^X$ we have a more precise upper bound on the (network) weight $nw(\F)=w(\F)$.

\begin{proposition}\label{p3.2} For any topological space $X$, each compact function family $\F\subset\IR^X$ has weight $w(\F)\le \dec_{cl}(\F)\cdot h_{cl}kd(X)\cdot s(X)$.
\end{proposition}

\begin{proof} Let $\C$ be a cover of $X$ by closed or finite subsets such that $|\C|=\dec_{cl}(\F)$ and for each $C\in\C$ and $f\in\F$ the restriction $f|C$ is continuous. Let $Y=\oplus\C$ be the topological sum of the spaces $C\in\C$ and $\pi:Y\to X$ be the natural projection, which induces a linear topological embedding $\pi^*:\IR^X\to\IR^Y$, $\pi^*:f\mapsto f\circ \pi$. The choice of the cover $\C$ guarantees that $\pi^*(\F)\subset C_p(Z)$. For each space $C\in\C$ choose a family $\K_C$ of compact subsets such that $|\K_C|=kd(C)$ and $\cup\K_C$ is dense in $C$. Since $C$ is closed or finite (and hence compact), we get $kd(C)\le h_{cl}kd(X)$. Consider the topological sum $Z=\bigoplus\limits_{C\in\C}\bigoplus\K_C$ and the natural projection $p:Z\to Y$, which induces a linear continuous operator $p^*:C_p(Y)\to C_p(Z)$, $p^*:f\mapsto f\circ p$. Since $p(Z)$ is dense in $Y$, the operator $p^*$ is injective. By the compactness of the function family $\F$, the restriction $p^*\circ \pi^*|\F:\F\to C_p(Z)$, being injective and continuous, is a topological embedding. Consequently, the compact function family $\F\subset\IR^X$ is homeomorphic to the compact function family $\F_Z=p^*\circ\pi^*(\F)\subset C_p(Z)$.

For each compact subset $K\in\K_C$, consider the restriction operator $R_K:C_p(Z)\to C_p(K)$, $R:f\mapsto f|K$, and let $\F_K=R_K(\F_Z)\subset C_p(K)$.
Since $Z=\bigoplus_{C\in\C}\bigoplus\K_C$, the restriction operators $R_K$, $K\in\K_C$, $C\in\C$, compose a topological embedding $R:C_p(Z)\to\prod_{C\in\C}\prod_{K\in\K_C}C_p(K)$. Consequently,
$$w(\F)=w(\F_Z)\le\sum_{C\in\C}\sum_{K\in\K_C}w(\F_K).$$

So, it remains to evaluate the weight $w(\F_K)$ of the compact function family $\F_K\subset C_p(K)$ for each compact subset $K\in\K_C$, $C\in\C$. Observe that the compact space $K$ is homeomorphic to a compact subset of the space $X$ and hence has spread $s(K)\le s(X)$.

The compact subset $\F_K\subset C_p(K)$ induces a separately continuous function $\delta:\F_K\times K\to \IR$, $\delta:(f,x)\mapsto f(x)$, and a continuous function $\delta_\cdot:K\to C_p(\F_K)$ assigning to each point $x\in K$ the function $\delta_x:\F_K\to \IR$, $\delta_x:f\mapsto f(x)=\delta(f,x)$.
Then the image $E_K=\delta_\cdot(K)\subset C_p(\F_K)$ is an Eberlein compact with spread $s(E_K)\le s(K)\le s(X)$ and cellularity $c(E_K)\le s(E_K)\le s(X)$. Since the weight of Eberlein compacta is equal to their cellularity \cite[III.5.8]{Arh}, we conclude that $w(E_K)=c(E_K)\le s(X)$.

The surjective map $\delta_\cdot:K\to E_K$ induces a linear topological embedding $\delta_\cdot^*:C_p(E_K)\to C_p(K)$ such that $\F_K\subset \delta_\cdot^*(C_p(E_K))$. Then $w(\F_K)=nw(\F_K)\le nw(C_p(E_K))\le nw(E_K)\le s(X)$.

Finally,
$$
\begin{aligned}
w(\F)&=w(\F_Z)\le\sum_{C\in\C}\sum_{K\in\K_C}w(\F_K)\le\\
 &\le\sum_{C\in\C}|\K_C|\cdot s(X)=\sum_{C\in\C}kd(C)\cdot s(X)\le|\C|\cdot h_{cl}kd(X)\cdot s(X)=\dec_{cl}(\F)\cdot h_{cl}kd(X)\cdot s(X).
\end{aligned}
 $$
\end{proof}

\begin{corollary}\label{c3.3} For any compact topological space $X$, each compact function family $\F\subset\IR^X$ has weight $w(\F)\le \dec_{cl}(\F)\cdot s(X)$.
\end{corollary}

\section{Scatteredly continuous and weakly discontinuous function families}

A family of
functions $\F\subset Y^X$ from a topological space $X$ to a
topological space $Y$ is called
\begin{itemize}
\item {\em continuous} if each function $f\in \F$ is continuous;
\item {\em scatteredly continuous} if each non-empty subset
$A\subset X$ contains a point $a\in A$ at which each function
$f|A:A\to Y$, $f\in\F$, is continuous;
\item {\em weakly discontinuous}  if each non-empty subset $A\subset X$ contains
an open dense subspace $U\subset A$ such that each function
$f|U:U\to Y$, $f\in\F$, is continuous.
\end{itemize}

The following simple characterization can be derived from the
corresponding definitions and Theorem~4.4 of \cite{BB} (saying that
each scatteredly continuous function with values in a regular
topological space is weakly discontinuous).

\begin{proposition}\label{p4.1} A function family $\F\subset Y^X$
is scatteredly continuous (resp. weakly discontinuous) if and only
if so is the function $\Delta\F:X\to Y^\F$, $\Delta\F:x\mapsto\big(f(x)\big)_{f\in\F}$.
Consequently, for a regular topological space $Y$, a function family
$\F\subset Y^X$ is scatteredly continuous if and only if it is weakly discontinuous.
\end{proposition}

We say that two spaces $X,Z$ are {\em weakly homeomorphic} if there is a bijective map $f:X\to Y$ such that the functions $f$ and $f^{-1}$ are weakly discontinuous.

In fact, each weakly discontinuous function family $\F\subset \IR^X$ is homeomorphic to a continuous function family over a space $Z$ which is weakly homeomorphic to $X$.
The space $Z$ can be canonically constructed using the family $\F$. This can be done as follows.

For a subset $A\subset X$ let $\F|A=\{f|A:f\in\F\},$ $D(\F|A)$ be the set of points at which the function $\Delta\F|A$ is discontinuous, and $\bar D(\F|A)$ be the closure of the set $D(\F|A)$ in $A$.
The weak discontinuity of $\F$ guarantees that the set $\bar D(\F|A)$ is nowhere dense in $A$ and hence $\bar D(\F|A)\ne A$ if $A\ne\emptyset$.

Let $D_0(\F)=X$ and by transfinite induction for each ordinal $\alpha>0$ define a closed subset $D_\alpha(\F)$ of $X$ letting
$$D_\alpha(\F)=\bigcap_{\beta<\alpha}\bar D(\F|D_\beta(\F)).$$  Since $(D_\alpha(\F))_{\alpha}$ is a decreasing sequence of closed subspaces of $X$ and $D_{\alpha+1}(\F)\ne D_\alpha(\F)$ if $D_\alpha(\F)\ne\emptyset$, there is an ordinal $\alpha$ such that $D_\alpha(\F)=\emptyset$. The smallest ordinal with $D_{\alpha}(\F)=\emptyset$ if denoted by $\wid(\F)$ and called the {\em index of weak discontinuity} of $\F$, see \cite[\S4]{BB}. It is clear that $|\wid(\F)|\le hl(X)$.

The topological sum
$$X_\F=\bigoplus\limits_{\alpha<\wid(\F)}D_{\alpha}(\F)\setminus
D_{\alpha+1}(\F)$$ is called {\em the resolution space} of the
weakly discontinuous family $\F\subset SC_p(X)$. The natural
projection $\pi_\F:X_\F\to X$ is bijective and continuous, while its
inverse $\pi^{-1}_\F:X\to X_\F$ is weakly discontinuous. So, the
spaces $X$ and $X_\F$ are weakly homeomorphic.

Moreover, the bijective map $\pi:X_\F\to X$ induces a linear topological isomorphism $\pi^*_\F:\IR^X\to \IR^{X_\F}$, $\pi^*:f\mapsto f\circ\pi$, such that $\pi^*_\F(\F)\subset C_p(X_\F)$. So, the weakly discontinuous function family $\F$ is homeomorphic to the continuous function family $\pi^*_\F(\F)\subset C_p(X_\F)$ on the resolution space $X_\F$.

This fact will be used to find some upper bounds on the (closed) decomposition number $\dec(\F)$ (resp. $\dec_c(\F)$) of a weakly discontinuous function family $\F\subset SC_p(X)$.

\begin{proposition}\label{p4.2} For any (regular) topological space $X$ and each weakly discontinuous function family $\F\subset SC_p(X)$ we get $\dec(\F)\le |\wid(\F)|\le hl(X)$ (and $\dec_{cl}(\F)\le hl(X)$).
\end{proposition}

\begin{proof} The upper bound $\dec(\F)\le |\wid(\F)|\le hl(X)$ follows from the fact that for each ordinal $\alpha$ and each function $f\in\F$ the restriction $f|D_\alpha(\F)\setminus D_{\alpha+1}(\F)$ is continuous. If the space $X$ is regular, then each set $F=D_\alpha(\F)\setminus D_{\alpha+1}(\F)$, being a difference of two closed subsets of $X$, can be written as the union of $\le l(F)\le hl(X)$ many closed subsets of $X$, which implies that $\dec_{cl}(\F)\le|\wid(\F)|\cdot hl(X)=hl(X)$.
\end{proof}

Proposition~\ref{p4.2} combined with Propositions~\ref{p3.1} and \ref{p3.2} implies:

\begin{corollary}\label{c4.3} For any topological space $X$, each weakly discontinuous function family $\F\subset \IR^X$ has network weight $$nw(\F)\le \dec(\F)\cdot nw(X)\le hl(X)\cdot nw(X)=nw(X).$$
\end{corollary}

\begin{corollary} For any regular topological space $X$, each weakly discontinuous compact subspace $\F\subset \IR^X$ has weight $$w(\F)\le \dec_{cl}(\F)\cdot h_{cl}kd(X)\cdot s(X)\le hl(X)\cdot h_{cl}kd(X).$$
\end{corollary}

\begin{corollary} For any compact Hausdorff space $X$, each weakly discontinuous compact subspace $\F\subset \IR^X$ has weight $w(\F)\le hl(X)$.
\end{corollary}

\section{$\infty$-Convex sets in linear topological spaces}\label{s:convex}

In this section we shall establish some basic properties of  $\infty$-convex and $\sigma$-convex sets in linear topological spaces. Let us recall that a subset $B$ of a linear topological space $L$ is {\em $\sigma$-convex} (resp. {\em $\infty$-convex}) if for each (bounded) sequence $(x_n)_{n\in\IN}$ of points of the set $B$ and each sequence $(t_n)_{n\in\IN}$ of non-negative real numbers with $\sum_{n=1}^\infty t_n=1$ the series $\sum_{n=1}^\infty t_nx_n$ converges to some point of the set $B$.

\begin{proposition}\label{p5.1} A subset $B$ of a  linear topological space $X$ is $\sigma$-convex if and only if $B$ is bounded and $\infty$-convex.
\end{proposition}

\begin{proof} To prove this theorem, it suffices to check that each $\sigma$-convex subset $B\subset L$ is bounded. Assuming the converse, we can find a neighborhood $U\subset L$ of zero such that $B\not\subset tU$ for any positive real number  $t$. Then, for every $n\in\IN$ we can choose a point $x_n\in B\setminus 2^nU$ and observe that the series $\sum_{n=1}^\infty 2^{-n}x_n$ is divergent since $$\sum_{n=1}^{m}2^{-n}x_n-\sum_{n=1}^{m-1}2^{-n}x_n=2^{-m}x_m\notin U$$ for all $m\in\IN$.
\end{proof}

A subset $B$ of a linear topological space $L$ is called {\em sequentially complete} if each Cauchy
sequence in $B$ is convergent.

\begin{proposition} A convex subset $C$ of a locally convex linear topological space $L$ is $\infty$-convex if each closed bounded convex subset of $C$ is sequentially complete.
\end{proposition}

\begin{proof} Given a bounded sequence $(x_n)_{n\in\IN}$ in $C$ and a sequence $(t_n)_{n\in\IN}$ of non-negative real numbers with $\sum_{n=1}^\infty t_n=1$ we need to check that the series $\sum_{n=1}^\infty t_nx_n$ converges in $C$. Let $B$ be the closed convex hull of the bounded set $\{x_n\}_{n\in\IN}$ in $C$. The local convexity of $L$ implies that $B$ is bounded in $L$, and by our assumption, $B$ is sequentially complete.

For every $n\in\IN$ consider the point $y_n=(1-\sum_{k=1}^nt_k)x_1+\sum_{k=1}^nt_kx_k\in B$. We claim  that the sequence $(y_n)_{n=1}^\infty$ is Cauchy. Given a symmetric neighborhood $U=-U$ of zero in $L$, we need to find $N\in\IN$ such that $y_m-y_n\in U$ for all $m\ge n\ge N$. By the local convexity of the space $L$, there is an open convex neighborhood $V=-V$ of zero in $L$ such that $V+V\subset U$.

Since the set $B$ is bounded, there is a positive real number $\e>0$ so small that $\e\cdot B\subset V$. Then $[0,\e]\cdot B\subset V$ by the convexity of $V\ni 0$. Since $\sum_{k=1}^\infty t_k=1<\infty$, there is a number $N\in\IN$ such that $\sum_{k=N}^\infty t_k<\e$.

We claim that $y_m-y_n\in U$ for every $m\ge n\ge N$. Observe that
$y_m-y_n=-\sum_{k=n+1}^mt_kx_1+\sum_{k=n+1}^mt_kx_k$. If $\sum_{k=n+1}^mt_k=0$, then $y_m-y_n=0$ and we are done. So, assume that $s=\sum_{k=n+1}^mt_k>0$ and put $s_k=t_k/s$ for $n<k\le m$. Observe that
$s=\sum_{k=n+1}^mt_k\le\sum_{k=N}^\infty t_k<\e$ and
the point
$$\sum_{k=n+1}^mt_kx_k=s\sum_{k=n+1}^ms_kx_k$$ belongs to $s\cdot B$ by the convexity of the set $B$.
Then $y_m-y_n=-sx_1+\sum_{k=n+1}^mt_kx_k\in -sB+sB\subset -[0,\e]B+[0,\e]B\subset -V+V\subset U$. So, the sequence $(y_k)_{k=1}^\infty$ is Cauchy and by the sequential completeness of $B$,
this sequence converges to some point $b\in B$, equal to the sum of the series $\sum_{k=1}^\infty t_kx_k$.
\end{proof}

\begin{corollary} Each countably compact convex subset of a locally convex linear topological space is sequentially complete and $\sigma$-convex.
\end{corollary}

\section{$\infty$-Convex subsets in function spaces $SC_p^*(X)$}

The following theorem combined with Corollary~\ref{c4.3} implies Theorem~\ref{t1.1} announced in the Introduction and generalizes Theorem 2 of \cite{BBK}.

\begin{theorem}\label{t6.1} If $X$ is a topological space with countable tightness, then each $\infty$-continuous subset $\F\subset SC_p^*(X)$ is weakly discontinuous.
\end{theorem}

\begin{proof} By Proposition~\ref{p4.1}, the weak discontinuity of
the function family $\F$ is equivalent to the weak discontinuity
of the function $\Delta\F:X\to\IR^\F$, $\Delta\F:x\mapsto(f(x))_{f\in\F}$.
Since the space $X$ has countable tightness, the weak discontinuity
of $\Delta\F$ will follow from Proposition~2.3 of
\cite{BB} as soon as we check that  for each countable subset
$Q\subset X$ the restriction
$\Delta\F|Q:Q\to\IR^\F$ has a continuity point.

Assume conversely that the function $\Delta\F|Q$ is everywhere discontinuous.
Choose any function $f_0\in \F$ and find an open dense subset $U\subset Q$ such that $f_0|U$ is continuous. Let $U=\{x_n:n\in\IN\}$ be an enumeration of the countable set $U$.

Since the function $\Delta\F|U$ is everywhere discontinuous,
for each point $x_n\in U$ we can choose a function $g_n\in
\F$ such that the restriction $g_n|U$ is discontinuous at $x_n$. Since the function $g_n$ has bounded norm $\|g_n\|=\sup_{x\in X}|g_n(x)|$, we can choose a positive real number $\lambda_n\in(0,1]$ such that the function
$f_n=(1-\lambda_n)f_0+\lambda_ng_n$ has norm $\|f_n\|\le \|f_0\|+1$.
The continuity of $f_0|U$ and discontinuity of $g_n|U$ at $x_n$ imply that the function $f_n|U$ is discontinuous at $x_n$.

Observe that a function $f:U\to\IR$ is discontinuous at a point
$a\in U$ if and only if it has strictly positive oscillation
$$\osc_a(f)=\inf_{O_a}\sup\{|f(x)-f(y)|:x,y\in O_a\}$$at the point $a$.
In this definition the infimum is taken over all neighborhoods $O_a$
of $a$ in $U$.

We shall inductively construct a sequence $(t_n)_{n=1}^\infty$ of
positive real numbers such that for every $n\in\IN$ the following
conditions are satisfied:
\begin{enumerate}
\item[1)] $t_1\le \frac12$, $t_{n+1}\le \frac12t_n$, and $t_{n+1}\cdot\|f_{n+1}\|\le \frac12t_n\cdot\|f_n\|$,
\item[2)] the function $s_n=\displaystyle\sum_{k=1}^nt_kf_k$ restricted to $U$ is discontinuous at $x_n$,
\item[3)] $t_{n+1}\cdot\|f_{n+1}\|\le\frac18\osc_{x_n}(s_n|U)$.
\end{enumerate}

We start the inductive construction letting $t_1=1/2$. Then the
function $s_1|U=t_1\cdot f_1|U$ is discontinuous at $x_1$ by the
choice of the function $f_1$. Now assume that for some $n\in\IN$
positive numbers $t_1\dots,t_n$ has been chosen so that the function
$s_n=\displaystyle\sum_{k=1}^nt_kf_k$ restricted to $U$ is
discontinuous at $x_n$.

Choose any positive number $\tilde t_{n+1}$ such that
$$\tilde t_{n+1}\le \frac12t_n,\;\;\tilde t_{n+1}\cdot\|f_{n+1}\|\le\tfrac
12t_n\cdot\|f_{n}\|\mbox{ \ and \ }  \tilde
t_{n+1}\cdot\|f_{n+1}\|\le\tfrac18\osc_{x_n}(s_n|U),$$ and consider
the function $\tilde s_{n+1}=s_n+\tilde t_{n+1}f_{n+1}$. If the
restriction of this function to $U$ is discontinuous at the point
$x_{n+1}$, then put $t_{n+1}=\tilde t_{n+1}$ and finish the
inductive step. If $\tilde s_{n+1}|U$ is continuous at $x_{n+1}$,
then put $t_{n+1}=\frac12\tilde t_{n+1}$ and observe that the
restriction of the function
$$s_{n+1}=\displaystyle\sum_{k=1}^{n+1} t_kf_k=s_n+\tfrac12\tilde t_{n+1}f_{n+1}=
\tilde s_{n+1}-\tfrac12\tilde t_{n+1}f_{n+1}$$ to $U$ is discontinuous
at $x_{n+1}$. This completes the inductive construction.
\smallskip

The condition (1) guarantees that  $\displaystyle\sum_{n=1}^\infty
t_n\le  1$ and hence the number
$t_0=1-\displaystyle\sum_{n=1}^\infty t_n$ is non-negative. So, we can consider the function
$$s=\displaystyle\sum_{n=0}^\infty t_nf_n$$ which is well-defined and belongs to
$\F$ by the $\infty$-convexity of $\F$.

We claim that the function $s\in\F\subset SC_p(X)$ has everywhere discontinuous restriction $s|U$. Assume conversely that $s|U$ is continuous at some point  $x_n\in U$. Observe
that $$s=t_0f_0+s_n+\displaystyle\sum_{k=n+1}^\infty t_{k}f_k$$and
hence
$$s_n=s-t_0f_0-\displaystyle\sum_{k=n+1}^\infty t_kf_k=s-t_0f_0-F_n,$$
where $F_n=\displaystyle\sum_{k=n+1}^\infty t_kf_k$. The conditions
(1) and (3) of the inductive construction guarantee that the
function $F_n$
 has norm
$$\|F_n\|\le\displaystyle\sum_{k=n+1}^\infty t_k\|f_k\|\le 2 t_{n+1}\|f_{n+1}\|\le
\frac14\osc_{x_n}(s_n|U).$$
Since $s_n=s-t_0f_0-F_n$, the triangle inequality implies that
$$0<\osc_{x_n}(s_n|U)\le \osc_{x_n}(s|U)+\osc_{x_n}(t_0f_0|U)+
\osc_{x_n}(F_n)\le$$ $$\le 0+0+2\|F_n\|\le\frac12\osc_{x_n}(s_n|U).$$
This contradiction shows that $s|U$ is everywhere discontinuous, which is not possible as $s\in\F\subset SC_p(X)$ by the $\infty$-convexity of the family $\F$.
So, the family $\F$ is weakly discontinuous.
\end{proof}

Theorem~\ref{t6.1} combined with Propositions~\ref{p3.1}, \ref{p3.2} and \ref{p4.2} imply:

\begin{corollary} If $X$ is a (regular) topological space with countable tightness, then each  $\infty$-convex subset $\F\subset SC_p^*(X)$ has (closed) decomposition number ($\dec_{cl}(\F)\le hl(X)$) $\dec(\F)\le hl(X)$ and network weight $nw(\F)\le nw(X)$.
\end{corollary}

\begin{corollary}  If $X$ is a regular topological space with countable tightness, then each compact  convex subset $\F\subset SC_p^*(X)$ has weight $w(\F)\le hl(X)\cdot h_{cl}kd(X)$.
\end{corollary}

\begin{corollary}  If $X$ is a compact Hausdorff topological space with countable tightness, then each compact  convex subset $\F\subset SC_p^*(X)$ has weight $w(\F)\le hl(X)$.
\end{corollary}

\section{Potentially bounded sets}

In order to generalize Theorem~\ref{t6.1} to $\infty$-convex sets in
the function spaces $SC_p(X)$ we need to introduce and study the
notion of a potentially bounded set.

Let us recall that a subset $B$ of a linear topological space $L$ is {\em bounded} if for each neighborhood $U\subset L$ of zero there is a positive real number $r$ such that $B\subset rU$. Observe that a family of functions $\F\subset\IR^X$ is bounded if and only if for each point $x\in X$ the set
$\F(x)=\{f(x):f\in\F\}$ is bounded in the real line.

A subset $B$ of a linear topological space $L$ is {\em potentially bounded} if for each sequence $(x_n)_{n\in\w}$ of points of $B$ there is a sequence $(t_n)_{n\in\w}$ of positive real numbers such that the set $\{t_nx_n\}_{n\in\w}$ is bounded in $L$.

Observe that for any topological space $X$ the space $l_\infty(X)$
of bounded real-valued functions on $X$ is potentially bounded in
$\IR^X$. Consequently, the space $SC^*_p(X)$ of all bounded
scatteredly continuous functions also is potentially bounded in
$\IR^X$.

\begin{proposition}\label{p7.1} Let $X$ be a $\sigma$-compact topological space. Each function family $\F\subset\IR^X$ with $\dec_{cl}(\F)\le\aleph_0$ is potentially bounded in $\IR^X$.
\end{proposition}

\begin{proof} Since $\dec_{cl}(\F)\le\aleph_0$, there is a countable closed cover $\C$ of $X$ such that for each $C\in\C$ and $f\in\F$ the restriction $f|C$ is continuous. Since the space $X$ is $\sigma$-compact, we can additionally assume that each set $C\in\C$ is compact.

To show that the set $\F$ is potentially bounded in $\IR^X$, fix any sequence $(f_n)_{n\in\w}$ in $\F$. Let $\C=\{C_n:n\in\w\}$ be any enumeration of the family $\C$. For every $n\in\w$ choose a positive real number $t_n$ such that $t_n\cdot \max |f_n(C_0\cup\dots\cup C_n)|<1$. We claim that the set $\{t_nf_n\}_{n\in\w}$ is bounded in $\IR^X$.

Indeed, for each point $x\in X$ we can find $k\in\w$ with $x\in C_k$ and observe that $$|t_nf_n(x)|\le t_n\max|f_n(C_0\cup\dots\cup C_n)|\le 1$$ for every $n\ge k$, which implies that the set $\{t_nf_n(x)\}_{n\in\w}\subset \{t_nf_n(x)\}_{n=0}^k\cup[-1,1]$ is bounded.
\end{proof}

Propositions~\ref{p7.1} and \ref{p4.2} imply:

\begin{corollary}\label{c7.2} Let $X$ be a regular $\sigma$-compact topological space. Each weakly discontinuous function family $\F\subset SC_p(X)$ is potentially bounded in $\IR^X$.
\end{corollary}

The $\sigma$-compactness of $X$ is essential in Corollary~\ref{c7.2}.

\begin{example} The set $\{\pi_n\}_{n\in\w}$ of coordinate projections $\pi_n:\w^\w\to\w$, $\pi_n:(x_k)_{k\in\w}\mapsto x_k$, is not potentially bounded in the function space $C_p(\w^\w)$.
\end{example}

\begin{proof} Assuming that the set $\{\pi_n\}_{n\in\w}$ is potentially bounded, we could find a sequence of positive real number $(\e_n)_{n\in\w}$ such that the set $\{\e_n\pi_n\}_{n\in\w}$ is bounded in $C_p(\w^\w)$. Choose any increasing number sequence $x=(x_n)_{n\in\w}\in\w^\w$ such that $\lim_{n\to\infty}\e_n\cdot x_n=\infty$ and observe that for the point $x$ the set $\{\e_n\pi_n(x)\}_{n\in\w}=\{\e_nx_n\}_{n\in\w}\subset \IR$ is not bounded, which means that the function sequence $\{\e_n\pi_n\}_{n\in\w}$ is not bounded in $C_p(\w^\w)$.
\end{proof}

\section{Spaces with countable strong fan tightness and related spaces}\label{s:fan}

In this section we shall discuss spaces with countable strong fan tightness and their relation to $R$-separable and scatteredly $R$-separable spaces.

Let us recall \cite{Sakai} that a topological space $X$ has {\em countable strong fan tightness} if for any sequence $(A_n)_{n\in\w}$ of subsets of $X$ and a point $a\in\bigcap_{n\in\w}\bar A_n$ in the intersection of their closures, there is a sequence of points $a_n\in A_n$, $n\in\w$, such that the set $\{a_n\}_{n\in\w}$ contains the point $a$ in its closure.

It is easy to see that each first countable space has countable strong fan tightness.

Following \cite{BBM} we say that a topological space $X$ is {\em $R$-separable} if for each sequence $(D_n)_{n\in\w}$ of dense subsets of $X$ there is a sequence of points $x_n\in D_n$, $n\in\w$, such that the set $\{x_n\}_{n\in\w}$ is dense in $X$. We shall say that a topological space $X$ is {\em countably $R$-separable} if each countable subspace of $X$ is $R$-separable.

Proposition 52 of \cite{BBM} implies:

\begin{proposition} Each topological space with countable strong fan tightness is countably $R$-separable.
\end{proposition}

Our next proposition follows from Theorem 48 of \cite{BBM}.

\begin{proposition} A topological space $X$ is countably $R$-separable if each countable subspace $Z$ of $X$ has character $\chi(Z)<\cov(\M)$.
\end{proposition}

Here $\cov(\M)$ stands for the smallest cardinality of a cover of the real line $\IR$ by nowhere dense subsets.

In fact, we shall need a bit weaker notion than the countable $R$-separability, called the scattered $R$-separability.

We define a topological space $X$ to be {\em scatteredly $R$-separable} if for each countable subspace $Z\subset X$ without isolated points and each sequence $(D_n)_{n\in\w}$ of dense subsets in $Z$ there is a sequence of points $z_n\in D_n$, $n\in\w$, forming a non-scattered subspace $\{z_n\}_{n\in\w}$.
It is clear that each countably $R$-separable space is scatteredly $R$-separable, so we get the following chain of implications:
\smallskip

\centerline{first countable $\Rightarrow$ countable strong fan tightness $\Rightarrow$ countably $R$-separable $\Rightarrow$ scatteredly $R$-separable.}
\smallskip

\begin{problem} Is each scatteredly $R$-separable space countably $R$-separable?
\end{problem}

\section{Potentially bounded $\infty$-convex subsets in function spaces $SC_p(X)$}

Theorem~\ref{t6.1} established the weak discontinuity of
$\infty$-convex subsets of the function spaces $SC_p^*(X)$. The
following theorem does the same for function spaces $SC_p(X)$.

\begin{theorem}\label{t9.1} If a topological space $X$ has countable tightness and is scatteredly $R$-separable, then each potentially bounded $\infty$-convex subset $\F\subset SC_p(X)$ is weakly discontinuous and has network weight $nw(\F)\le nw(X)$.
\end{theorem}

\begin{proof} By Proposition~\ref{p4.1}, the weak discontinuity of the family $\F$ is equivalent to the weak discontinuity of the diagonal product $\Delta\F:X\to\IR^\F$, $\Delta\F:x\mapsto (f(x))_{f\in\F}$. Since the space $X$ has countable tightness, the weak discontinuity of the map $\Delta\F$ will follow as soon as we check that for each countable subset $Q\subset X$ the restriction $\Delta\F|Q$ has a point of continuity. Assume conversely that for some countable subset $Q\subset X$ the restriction $\Delta\F|Q$  has no continuity points. Then the set $Q$ has no isolated points. Let $g_0\in\F$ be any function. Since $g_0$ is weakly discontinuous, there is an open dense set $U\subset Q$ such that the restriction $g_0|U$ is continuous.

Let $U=\{x_n\}_{n\in\IN}$ be an enumeration of the set $U$. Since
the function $\Delta\F|U$ is everywhere discontinuous, for every
$n\in\IN$ we can choose a function $f_n\in \F$ whose restriction
$f_n|U$ is discontinuous at the point $x_n$. Since the family $\F$
is potentially bounded, there is a number sequence
$(\lambda_n)_{n\in\IN}\in(0,1]^\w$ such that the sequence
$\{\lambda_n f_n\}_{n\in\IN}$ is bounded in $SC_p(X)$. Then the set
of functions $g_n=(1-\lambda_n)g_0+\lambda_n f_n$, $n\in\w$, is
bounded and belongs to the convex set $\F$. Moreover, each function
$g_n|U$ is discontinuous at $x_n$.   The boundedness of the function
sequence $(g_n)_{n\in\IN}$ guarantees that
$\sup_{n\in\IN}|g_n(x)|<\infty$ for all $x\in X$.
\smallskip

Now we consider two cases.
\smallskip

1. There is a non-empty open subset $V\subset U$ such that for every $n\in\IN$ the function $g_n|V$ is bounded. Repeating the argument of the proof of Theorem~\ref{t6.1}, we can find a sequence of positive real numbers $(t_n)_{n\in\w}$ with $\sum_{n=1}^\infty t_n=1$ such that the function $g=\sum_{n=1}^\infty t_ng_n$ has everywhere discontinuous restriction $g|V$, which is not possible as $g\in\F\subset SC_p(X)$ by the $\infty$-convexity of $\F$.
\smallskip

2. For each non-empty open subset $V\subset U$ there is $n\in\IN$ such that the function $g_n|V$ is unbounded. Then for each $m\in\IN$ the set $D_m=\{x\in U:\sup_{n\in\IN}|g_n(x)|>m\}$ is dense in $U$.

Since the space $X$ is scatteredly $R$-separable and the countable space $U\subset X$ has no isolated points, there is a non-scattered subspace $\{x_k\}_{k\in\IN}\subset U$ such that $x_k\in D_{16^k}$ for every $k\in\IN$. Since $16^k<\sup_{n\in\IN}|g_n(x_k)|<\infty$, for every $k\in\IN$ we can choose a number $n_k\in\IN$ such that $|g_{n_k}(x_k)|>\max\{16^k, \frac67\sup_{n\in\IN}|g_n(x_k)|\}$.

By induction we shall construct a sequence of real numbers $(t_k)_{k\in\w}\in[0,1]^\w$ such that
\begin{enumerate}
\item $\frac1{8^k}\le t_k\le \frac2{8^k}$ for all $k\in\IN$;
\item $|\sum_{i=1}^kt_ig_{n_i}(x_k)|\ge \frac1{2\cdot 8^k}{|g_{n_k}(x_k)|}$ for every $j\le k$;
\end{enumerate}
We start the inductive construction letting $t_1=\frac12$. Now assume that for some number $k>1$ the real numbers $t_1,\dots,t_{k-1}$ satisfying the conditions (1), (2) have been constructed.
Consider the function $s_{k-1}=\sum_{i=1}^{k-1}t_ig_{n_i}$. Observe that the interval
$s_{k-1}(x_k)+g_{n_k}(x_k)[\frac1{8^k},\frac2{8^k}]$ has length $\frac1{8^k}{|g_{n_k}(x_k)|}$ and hence there is a number $t_k\in[\frac1{8^k},\frac2{8^k}]$ such that $|s_{k-1}(x_k)+t_kg_{n_k}(x_k)|\ge \frac1{2\cdot 8^k}{|g_{n_k}(x_k)|}$.

Since the set $\{g_{n_k}\}_{k\in\IN}$ is bounded and $\sum_{k=1}^\infty t_k\le 1$, the function $s=\sum_{k=1}^\infty t_kg_{n_k}\in\IR^X$ is well-defined. Observe that for every $k\in\IN$ the choice of the function $g_{n_k}$ with $|g_{n_k}(x_k)|>\frac67\sup_{n\in\IN}|g_n(x_k)|$ implies
the lower bound:
$$
\begin{aligned}
|s(x_k)|&\ge|\sum_{i=1}^kt_ig_{n_i}(x_k)|-\sum_{i=k+1}^\infty t_i|g_{n_i}(x_k)|\ge \\ &\ge\frac1{2\cdot 8^k}{|g_{n_k}(x_k)|}-\sum_{i=k+1}^\infty\frac2{8^i}\frac76|g_{n_k}(x_k)|=\\
&=|g_{n_k}(x_k)|\Big(\frac1{2\cdot 8^k}-\frac{1}{3\cdot 8^{k}}\Big)=\frac{1}{6\cdot 8^k}|g_{n_k}(x_k)|>\frac{16^k}{6\cdot 8^k}.
\end{aligned}
$$
So, $s(x_k)\to\infty$ as $k\to\infty$, which implies that the function $s$ restricted to the non-scattered subset $\{x_k\}_{k\in\IN}$ is not weakly discontinuous. Then for the number $t_0=1-\sum_{k=1}^\infty t_k$  the function $t_0g_0+s$ also is not weakly discontinuous, which is not possible as $t_0f_0+s\in\F\subset SC_p(X)$ by the $\infty$-convexity of $\F$.

This contradiction shows that the function family $\F$ is weakly discontinuous and hence has network weight $nw(\F)\le nw(X)$ by Corollary~\ref{c4.3}.
\end{proof}

Combining Theorem~\ref{t9.1} with Proposition~\ref{p3.2} we deduce:

\begin{corollary}\label{c9.2} If a regular topological space $X$ has countable tightness and is scatteredly $R$-separable, then each compact convex subset $\F\subset SC_p(X)$ is weakly discontinuous and has weight $w(\F)\le hl(X)\cdot h_{cl}kd(X)$.
\end{corollary}

Let us recall that a topological space $X$ is {\em perfectly normal} if it is normal and each closed subset is of type $G_\delta$ in $X$. It is easy to see that a compact space $X$ is perfectly normal if and only if $hl(X)\le\aleph_0$. Since each perfectly normal compact space is first countable, it has countable tightness and countable strong fan tightness. This observation and Corollary~\ref{c9.2} imply

\begin{corollary} For a perfectly normal compact space $X$, each compact convex subset in the function space $SC_p(X)$ is metrizable.
\end{corollary}

\section{Some Open Problems}

Let us recall \cite{God} that a topological space $K$ is {\em Rosenthal compact} if $K$ is homeomorphic to a compact subspace of the space $\mathcal B_1(X)\subset\IR^X$ of functions of the first Baire class on a Polish space $X$. In this definition the space $X$ can be assumed to be equal to the space $\w^\w$ of irrationals.

\begin{problem}\label{pr4} Is each Rosenthal compact space $K$ homeomorphic to a compact subset of the function space $SC_p(\w^\w)$?
\end{problem}

According to Theorem~3 of \cite{BBK}, this problem has affirmative solution in the realm of zero-dimensional separable Rosenthal compacta.
A particularly interesting instance of Problem~\ref{pr4} concerns non-metrizable convex Rosenthal compacta. One of the simplest spaces of this sort is the Helly space. We recall that the {\em Helly space} is the subspace of $B_1(I)$ consisting of all non-decreasing functions $f:I\to I$ of the unit interval $I=[0,1]$.

\begin{problem} Is the Helly space homeomorphic to a compact subset of the function space $SC_p(\w^\w)$?
\end{problem}

%\newpage


\begin{thebibliography}{99}

\bibitem{Arh} A.V.~Arkhangel'skii, {\em Topological function spaces}, Moskov. Gos. Univ., Moscow, 1989. 223 pp. (in Russian).
%\emph{Архангельский А.В.} Топологические пространтсва функций. --- M.: Изд-во МГУ, 1989.

\bibitem{AB} A.V.~Arkhangel'skii, B.M.~Bokalo, {\em The tangency of topologies and tangential properties of topological spaces}, Trudy Moskov. Mat. Obshch. {\bf 54} (1992),
160--185, 278--279 (in Russian).

\bibitem{BB} T.~Banakh, B.~Bokalo, {\em  On scatteredly continuous maps between topological spaces},  Topology Appl. \textbf{157} (2010) 108--122.

\bibitem{BBK} T.~Banakh, B.~Bokalo, N.~Kolos, {\em  $\sigma$-convex sets in spaces of scatteredly
continuous functions}, Mat. Visnyk NTSh (submitted); available at http://arxiv.org/abs/1204.2438.

\bibitem{BLM} T.~Banakh, W.E.~Lyantse, Ya.V.~Mykytyuk, {\em $\infty$-convex sets and their applications to the proof of certain classical theorems of functional analysis}, Mat. Stud. {\bf 11}:1 (1999) 83--84.

\bibitem{BBM} A.~Bella, M.~Bonanzinga, M.~Matveev, {\em Variations of selective separability}, Topology Appl. {\bf 156}:7 (2009) 1241--1252.

\bibitem{BK1} B.~Bokalo, N.~Kolos, {\em When does $SC(X)=\mathbb{R}^X$ hold?} Topology. {\bf 48} (2009) 178--181.

\bibitem{BK2} B.~Bokalo, N.~Kolos, {\em On normality of spaces of scattredly continuous maps},   Mat.~Stud. {\bf 35}:2 (2011) 196--204.

%\emph{Бокало Б.М., Маланюк О.П.} Про майже неперервні відображення // Мат. Студ. --- 1998. --- \textbf{9}, №1. --- С. 90--93.

\bibitem{En} R.~Engelking, {\em General Topology}, PWN, Warsaw, 1977.
%\emph{Энгелькинг Р.} Общая топология. --- М.: Мир, 1986.

\bibitem{Juh} I.~Juhasz, {\em Cardinal functions in Topology}, Math. Centre Tracts \textbf{34}, Amsterdam, 1971.

\bibitem{God} G.~Godefroy, {\em Compacts de Rosenthal}, Pacific J. Math. {\bf 91}:2 (1980) 293--306.

\bibitem{Sakai} M.~Sakai, {\em Property C'' and function spaces}, Proc. Amer. Math. Soc. {\bf 104} (3) (1988) 917--919.

\end{thebibliography}
\end{document}